\documentclass[10pt,reqno]{amsart}

\usepackage{graphicx}
\usepackage{amssymb}
\usepackage{epsfig}
\usepackage{amsmath}
\usepackage{amsthm}
\usepackage{color}
\definecolor{r}{rgb}{0.9,0.3,0.1}
\definecolor{b}{rgb}{0.1,0.3,0.9}

\usepackage{amsmath}
\usepackage{amssymb}
\usepackage{amsthm}
\usepackage{enumerate}
\usepackage{amsbsy}
\usepackage{amsfonts}
\newtheorem{theo}{Theorem}[section]
\newtheorem{defin}[theo]{Definition}
\newtheorem{prop}[theo]{Proposition}
\newtheorem{coro}[theo]{Corollary}
\newtheorem{lemm}[theo]{Lemma}
\newtheorem{rem}[theo]{Remark}

\newcommand{\al}{\alpha}
\newcommand{\be}{\beta}

\newcommand{\Ga}{\Gamma}
\newcommand{\la}{\lambda}

\newcommand{\Om}{\Omega}
\newcommand{\ka}{\kappa}

\newcommand{\ep}{\epsilon }
\newcommand{\te}{\theta}
\newcommand{\De}{\Delta}

\newcommand{\pa}{\partial}

\newcommand{\R}{{\mathbb R}^n}

\newcommand{\ri}{\rightarrow}
\newcommand{\Rn}{{\mathbb R}^{n-1}}

\newcommand{\na}{\nabla}

\begin{document}
\baselineskip=18pt

\title[]{Boundary integral operator for the fractional Laplacian in the bounded  smooth domain}

\author{Tongkeun Chang}
\address{Department of Mathematics, Yonsei University \\
Seoul, 136-701, South Korea}
\email{chang7357@yonsei.ac.kr}

\subjclass{Primary 45P05 ;
Secondary 30E25}

\keywords{boundary integral operator, layer potential, fractional Laplacian, bounded  domain}

\begin{abstract}
We study  the  boundary integral operator induced from the fractional
Laplace equation in a bounded smooth domain.   For $\frac12 < \al <1$, we show the bijectivity of the boundary integral operator $S_{2\al}: L^p(\pa \Om) \ri H^{2\al -1}_p(\pa \Om), \, 1 < p < \infty$.
As an application, we show the existence of the solution of  the boundary value problem of the fractional Laplace
equation.
\end{abstract}

\subjclass{Primary 45P05 ;
Secondary 30E25}

\keywords{boundary integral operator, layer potential, fractional Laplacian, bounded  domain}

\maketitle

\section{Introduction}
\setcounter{equation}{0}
In this paper, we study a  boundary integral operator defined on the
boundary of a  bounded domain in $\R, \,\, n \geq 3$. Let $\Om$ be a bounded
smooth domain and $\Ga_{2\al}(x) = c(n,2\al) \frac1{|x|^{n-2\al}} $ be the Riesz kernel of order $2\al, \,\, 0 < 2\al < n$  in $\R$, where $c(n, 2\al)$ is the normalized constant.
The layer potential of a fractional Laplacian for  $\phi \in L^p(\pa
\Om), \,1 < p < \infty$ is defined by
\begin{align}\label{boundary integral operator}
{\mathcal S}_{2\al} \phi (x) = \int_{\pa \Om}   \Ga_{2\al} (x - Q) \phi(Q) dQ,\qquad x \in \R \setminus \pa \Om.
\end{align}
The boundary integral operator
\begin{align}\label{boundary-restriction}
S_{2\al}\phi = ({\mathcal S}_{2\al} \phi)|_{\pa \Om}
\end{align}
 is defined by the  restriction  of  ${\mathcal S}_{2\al} \phi$ over $\pa \Om$.

M. Z\" ahle studied the Riesz potentials in  a general
metric space $(X, \rho)$ with Ahlfors $d$-regular measure $\mu$. She
showed $S_{2\al}: L^2(X, d\mu) \ri L^2_{2\al}(X, d\mu),$  $0 < 2\al < n$
is invertible,  where $L^2(X, d\mu)$ is decomposed by the null space
$N(S_{2\al}) $ and  the orthogonal compliment of $N(S_{2\al})$, that is, $L^2(X,
d\mu) = N(S_{2\al})   \otimes L^2_{2\al}(X, d \mu)$ (see \cite{Z} and  \cite{Z2}).

The author  of \cite{C} showed that the boundary integral operator $S_{2\al}$ defined
in \eqref{boundary-restriction} is extended to $H_2^{-\al +\frac12}(\pa \Om), \,\,  \frac12 < \al <1$ such that
$
S_{2\al} : H_2^{-\al +\frac12}(\pa \Om) \ri H_2^{\al -\frac12}(\pa \Om)
$
is bijective operator and ${\mathcal S}_{2\al} \phi \in \dot H_2^\al(\R)$ for
$\phi \in H_2^{-\al +\frac12}(\pa \Om)$ (see section \ref{function spaces} for the definition of the function spaces).

When $2\al =2$, $\Ga_2$ is the fundamental solution of the Laplace equation in
$\R$ and \eqref{boundary integral operator} is a single layer
potential of the Laplace equation. The single layer potential  and
boundary layer potential of the Laplace equation  were studied by many
mathematicians to show the existence of the solution of the boundary value problem of the
Laplace equation in a bounded domain (see \cite{FJR}, \cite{FMM},
\cite{JK} and  \cite{V}).

The first   result of this paper is  the  following Theorem.

\begin{theo}\label{theo1}
Let  $\Om$ be a bounded $C^2$-domain in $\R, \,\, 3 \leq n$. Let $\frac12 < \al <1$ and $1 < p < \infty$.
Then,
$S_{2\al} : L^p(\pa\Om) \ri H_p^{2\al-1} (\pa \Om)$ is bijective.
\end{theo}
The function space $H_p^{2\al-1} (\pa \Om)$    is defined in the section \ref{function spaces}.

The layer potential for $\phi \in B^s_p(\pa \Om), \, s < 0, \, 1 < p < \infty$ is defined by
\begin{align}\label{def-single-dual}
{\mathcal S}_{2\al} \phi (x) = <\phi, \Ga_{2\al} (x - \cdot)>, \quad x \in \R \setminus \pa \Om,
\end{align}
where
$<\cdot, \cdot> $ is the duality paring between
$B^s_p (\pa \Om) $ and $B^{-s}_{p'}(\pa \Om), \,\, \frac1p + \frac1{p'} =1$.  In particular, if
$\phi \in L^p(\pa \Om)$, then ${\mathcal S}_{2\al} \phi$ is defined by
\eqref{boundary integral operator}.
The second result  is   the following Theorem.
\begin{theo}\label{theo0905}
Let   $ \frac12 < \al < 1$ and $1 < p < \infty$. For $\phi \in B_p^{s }(\pa \Om)$, let
$u = {\mathcal S}_{2\al} \phi$ be a layer potential defined in
\eqref{def-single-dual}.
Let  $-2\al +1 -\frac1p < s < 0$. Then, $u \in B^{s+2\al -1 +\frac1p}_{loc,p} (\R)$ and
\begin{align}\label{0904}
\| u\|_{B^{s + 2\al -1 +\frac1p}_p(B_R)} \leq c_R \| \phi \|_{B_p^{s}(\pa \Om)},
\end{align}
where $B_R$ is open ball in $\R$ whose radius is $R$ and center is origin such that $\Om \subset B_R$.
Moreover, if $p > \frac{n-1}{n+s -1}$, the $u \in \dot{B}^{s+2\al -1 +\frac1p}_{p} (\R)$ such that
\begin{align}\label{0904-2}
\| u\|_{\dot B^{s + 2\al -1 +\frac1p}_p(\R)} \leq c \| \phi \|_{B_p^{s}(\pa \Om)}.
\end{align}
\end{theo}
The function spaces $B^{s+2\al -1 +\frac1p}_{loc,p} (\R)$ and  $\dot{B}^{s+2\al -1 +\frac1p}_{p} (\R)$    are defined in the section \ref{function spaces}.

The boundary integral operators (the single layer potential and the double layer potential)
have been studied by
many mathematicians. The bijectivity of the  layer potential has been used
to show the existence of the  solutions of the partial differential equations in a
bounded   domain or bounded   cylinder (see
\cite{B1}, \cite{B2}, \cite{DKV}, \cite{FKV}, \cite{HL},  \cite{M}
and \cite{She}).

As in the vast literature, we apply  the bijectivity of the
boundary  integral operator to the boundary value problem of the
fractional Laplace equation in the bounded  smooth domain. The
fractional Laplacian of order $0 < \al< 1$ of a function $v: \R \ri
{\mathbb R}$ is expressed by the formula
\begin{align*}
\De^\al v (x) =  C(n,\al) \int_{\R} \frac{v(x+y) - 2v (x) + v(x-y)}{|x
-y|^{n+2\al}} dy,
\end{align*}
where $C(n,\al)$ is some normalization constant.
The fractional Laplacian can also be defined as a pseudo-differential operator
\begin{align*}
\widehat{(-\De)^\al  v}(\xi) = (2\pi  | \xi|)^{2\al} \hat{v} (\xi),
\end{align*}
where $\hat{ v}(\xi) :=\int_{\R} v(x) e^{-2\pi i \xi \cdot x } dx, \,\, \xi \in \R$
is the Fourier transform of $v$ in $\R$. In
particular, when $2\al =2$  it is  the Laplacian
  $\De v(x) = \sum_{1\leq i \leq n} \frac{\pa^2 }{\pa
x_i^2}v(x)$.

\begin{defin}\label{1013-2}
Let  $ 0 <\al <  1$. We say that $v $ is a {\it weak solution} of  $\De^\al u =0 $ in $\R \setminus \pa \Om$
if  $v$  satisfies
\begin{align}\label{1013}
\int_{\R}  v(x) \De^{\al} \psi(x) dx =
\int_{\R} (2\pi  |\xi|)^{2\al} \hat{v}(\xi) \bar {\hat{\psi}} (\xi) d \xi =0
\end{align}
for all $\psi \in C_c^\infty(\R \setminus \pa\Om)$.
\end{defin}
In fact, if $u$ is weak solution, then $u$ is  continuous function in $\R \setminus \pa \Om$ and satisfies
\begin{align*}
\De^\al u(x) =0 \qquad \mbox{for } \qquad x \in \R \setminus \pa \Om.
\end{align*}
(See Theorem 3.9  in \cite{BB}).

For the application of Theorem \ref{theo1} and Theorem \ref{theo0905}, we show the existence of the solution of the boundary value problem of
the fractional Laplace equation in the bounded smooth domain.
\begin{theo}\label{theo2-1}
Let $\Om$ be a bounded $C^2$ domain in $\R, \,\,  3 \leq n$ and $\frac12 < \al <1$.
Let  $0< t < 1$ and $1 < p < \infty$.
Then, for given $g \in B^t_p(\pa \Om)$,
the following equation
\begin{eqnarray}\label{main result}
\left\{\begin{array}{l}
 \De^\al u =0   \quad \mbox{in} \quad \R \setminus \pa \Om,\\
u|_{\pa\Om} = g  \in B^t_p(\pa \Om),\\
u \in B^{t +\frac1p}_{loc, p}(\R),\\
|u(x)|   = O(|x|^{-n +2\al} ) \quad \mbox{near}\quad \infty.
\end{array}\right.
\end{eqnarray}
has a  weak solution.
And, if $ \frac{n-1}{n+t -2\al}< p < \infty$, then $u \in \dot{B}_p^{t+\frac1p}(\R)$.
 Furthermore, there exists  $\phi \in B_p^{t-2\al
+1  }(\pa \Om)$ such that
\begin{align}\label{single layer solution}
u = {\mathcal S}_{2\al} \phi.
\end{align}
\end{theo}

The rest of this paper is organized as follows. In section
\ref{function spaces}, we introduce the several function spaces.
In section \ref{operator}, we introduce  the several property of the layer potential.
In section \ref{bijectivity-single}, we prove the Theorem \ref{theo1}.
In section \ref{proof-maintheo}, we   prove the Theorem \ref{theo0905}.
In section \ref{proof-maintheo-2}, we prove the Theorem \ref{theo2-1}.

\section{Function spaces}\label{function spaces}
\setcounter{equation}{0}

\subsection{Function spaces in $\R$.}
\label{inrn}
In this section, we introduce the several function spaces, Sobolev space and Besov space.
For $s \in {\mathbb R}$, we
consider a distribution  $G_s$ whose Fourier transform
in $\R$ is defined by
\begin{eqnarray*}
\widehat{ G_{s}} (\xi) = (1 + 4\pi^2 |\xi|^2 )^{-\frac{s}{2}}.
\end{eqnarray*}
For $s \in {\mathbb R}, \,\, 1\leq p \leq \infty$, we define the  Sobolev space $H_p^{s} (\R)$ by
\begin{eqnarray*}
H^s_p (\R) : &= \{ f \in {\mathcal S}'(\R) \, | \,
 \|f\|_{H^s_p (\R)} :  = \|  G_{-s} * f  \|_{L^p(\R)} < \infty\},
\end{eqnarray*}
where
$*$ is a convolution in $\R$ and ${\mathcal S}^{'}(\R)$
 is the dual space of the Schwartz space
${\mathcal S}(\R)$. In particular, when $s =k \in {\mathbb N} \cup \{ 0\}$ and $1 < p< \infty$,
\begin{align*}
 H^k_p(\R)=\{ f \, | \, D^\be f \in L^p(\R), \quad |\be| \leq k\},
\end{align*}
where $\be \in ({\mathbb N} \cup \{0\})^n$ and $|\be| = \be_1+ \be_2 + \cdots + \be_n$
for $\be= (\be_1, \be_2, \cdots, \be_n)$.

For  $k < s < k+1, \,\, k \in {\mathbb N}$,  we define the   Besov space
$B^{s}_{p} (\R)$ and  homogeneous Besov space $\dot B^{s}_{p} (\R)$ by
\begin{eqnarray*}
B^{s}_{p} (\R) = \{ f \in {\mathcal
S}^{'}(\R) \, | \, \|f\|_{B^{s}_{p}} <
\infty \, \}, \quad  \dot{B}^{s}_{p} (\R) = \{ f \in {\mathcal
S}^{'}(\R) \, | \, \|f\|_{\dot{B}^{s}_{p}} <
\infty \, \}
\end{eqnarray*}
with the norms
\begin{align*}
 \|f\|_{B^{s}_{p}}: & = \|f\|_{H^k_p} + \Big( \sum_{|\be| = k }\int \int_{\R \times \R}
  \frac{| D^\be f(x) - D^\be f(y)|^p}{|x-y|^{n + p(s-k)}}
  dydx \Big)^{\frac1p},\\
 \|f\|_{\dot{B}^{s}_{p}} :&   =  \Big( \sum_{|\be| = k }\int \int_{\R \times \R}
  \frac{| D^\be f(x) - D^\be f(y)|^p}{|x-y|^{n + p(s-k)}}
  dydx \Big)^{\frac1p}.
\end{align*}
If $s \in {\mathbb R}$ is negative, then we define $B^s_p$ and $\dot{B}^s_p$ as the dual spaces of $B^{-s}_{p'}$ and $\dot{B}^{-s}_{p'}$, respectively, where
$\frac1p + \frac1{p'} =1$.
The real interpolation method and complex interpolation method give
\begin{eqnarray*}
( H^{s_0}_p , H^{s_1}_p )_{\te, p} =
B^{s }_{p}, \qquad [H^{s_0}_p , H^{s_1}_p ]_{\te} =
H^{s }_{p}\end{eqnarray*}
for $ s = (1-\te) s_0 + \te s_1, \, s_0, \, s_1 \in {\mathbb R}, \quad
 0 < \te <1$ (see  theorem 6.4.5 in \cite{BL}).

\subsection{Function spaces in $\Om$.}
\label{inomega}
Let  $\Om$ be  bounded $C^2$-domain in $\R$.
Let   $R_{\Om} f$ be a restriction over $\Om$
 of the function $f$ defined in
 $\R$.  For $s \geq 0$, we define the function spaces
\begin{align*}
 H^{s}_p (\Om)
 : = \{ R_{\Om}f \, | \, f \in H^{s }_p(\Om) \}, \quad
 B^{s}_p (\R)
 : = \{ R_{\Om}f \, | \, f \in B^{s }_p(\R) \}
\end{align*}
with norms
\begin{align*}
\| f\|_{H^{s}_p (\Om) }:
= \inf \| F\|_{H^{s }_p(\R)  }, \quad
\| f\|_{B^{s}_p (\Om) }:
= \inf \| F\|_{B^{s }_p(\R)  },
\end{align*}
where infimums are taken in $ F \in H^{s}_p(\R)$ and $  F \in B^{s}_p(\R)$, respectively,
 such that $R_{\Om} F =f$.

Note that for non-negative integer $k$ and $1 < p < \infty$,
\begin{align}
H^{k}_p (\Om) =
\{ f \in L^p(\Om)\, | \, \sum_{ 0\leq l \leq k  }
\| D^{l} f\|_{L^p(\Om)} < \infty \}
\end{align}
and for $0 < \te < 1$,
\begin{align*}
 (H^{ k_1}_p (\Om), H^{ k_2}_p (\Om))_{\te, p}
= B^{s }_p (\Om), \qquad [H^{ k_1}_p (\Om), H^{ k_2}_p (\Om)]_\te
= H^{s }_p (\Om),
\end{align*}
where $s = \te k_1 + (1 -\te) k_2$  (see chapter 2 in \cite{JK}).
In particular, for  $k < s < k+1$, we have
\begin{align*}
\| f\|_{B^s_p(\Om)} \approx \| f\|_{H^k_p(\Om)} + ( \sum_{|\be | =k} \int_\Om \int_\Om \frac{|D^\be_x f(x) - D^\be_y f(y)|^p}{|x -y|^{n+ p(s -k)}} dxdy)^\frac1p.
\end{align*}
For negative $s \in {\mathbb R}$, we define $B^s_p(\Om)$ and $H^s_p(\Om)$ as the dual spaces of $B^{-s}_{p'}(\Om)$ and $H^{-s}_{p'}(\Om)$, respectively, where
$\frac1p + \frac1{p'} =1$.

For $s > 0$, we define the spaces $H^s_{p0}(\Om)$ and $B^s_{p 0}(\Om)$ as the closures of  $ C^\infty_c (\Om)$ in $H^s_{p}(\Om)$ and $B^s_{p}(\Om)$, respectively.

\subsection{Function spaces in $\pa \Om$.}
\label{inpaomega}
Let $\Om$ be a  bounded $C^2$-domain in $\R$.  Let us $\De(P,r) = B(P, r) \cap \pa \Om$ for $P \in \pa \Om$.
Then, there is $r_0 > 0$ such that for all $P \in \pa \Om$, there is  bijective  $C^2$-function $\Psi: B^{'}(0, r_0) \ri \De(P,r_0)$, where  $B^{'} (0, r_0)$ is open balls in $\Rn$ whose radius is $R$ and center is origin. Since $\Om$ is bounded domain, there are $P_1, \, P_2, \,\cdots, P_N$ such that $\pa \Om \subset \cup_{i} \De(P_i, r_0)$. Moreover, there are bijective $C^2$-functions $\Psi_i: B^{'}(0,r_0) \ri \De(P_i, r_0)$. Now, we say that $\phi$ is  in function space $H^s_p(\pa \Om), \, -2 \leq s \leq 2$ if  $ \phi \circ \Psi_i \in H^s_p(B^{'}(0, r_0))$ for all $1 \leq i \leq N$. And the norm is
\begin{align*}
\| \phi\|_{H^s_p(\pa \Om)} = \sum_{1 \leq i \leq N} \| \phi \circ \Psi_i\|_{H^s_p(B^{'}(0, r_0))}.
\end{align*}
Similarly, we define the function space $B^s_p(\pa \Om)$. Clearly, for $0 < s <2$, $H^{-s}_p(\pa \Om)$ and $B^{-s}_p(\pa \Om)$ are dual spaces of
$H^{s}_{p'}(\pa \Om)$ and $B^{-s}_{p'}(\pa \Om)$, $\frac1p + \frac1{p'} =1$, respectively.

For $0 < \te < 1$,
\begin{align}\label{0310}
(H^{ k_1}_p (\pa \Om), H^{ k_2}_p (\pa \Om))_{\te,p}
= B^{s }_p (\pa \Om), \qquad [H^{ k_1}_p (\pa \Om), H^{ k_2}_p (\pa \Om)]_\te
= H^{s }_p (\pa \Om),
\end{align}
where $s =(1- \te) k_1 + \te k_2$  (see chapter 2 in \cite{JK}).

\section{Boundary layer potential}\label{operator}
\setcounter{equation}{0}
Let $S_{2\al} $ be a boundary layer potential defined by \eqref{def-single-dual}. Because of $\pa \Om$ is $C^2$, we get the  following Proposition:
\begin{prop}\label{props-2}
Let $\Om$ be $C^2$ bounded domain.
For $ -2 \leq s \leq 3 -2\al $ and $1 < p < \infty$,
\begin{align}\label{dual}
S_{2\al} : H^s_p (\pa \Om) \ri H^{s + 2\al -1}_p(\pa \Om)
\end{align}
is bounded  operator. Moreover,  for $-1 \leq s \leq 0$,
\begin{align*}
S_2: H^s_p(\pa \Om) \ri H_p^{1 +s}(\pa \Om)
\end{align*}
is bijective.
\end{prop}
(See    \cite{V}).\\
Let $\psi, \,\, \phi \in C^2(\pa \Om)$ and $S^*_{2\al}: H^{-s - 2\al +1}_{p'}(\pa \Om) \ri   H^{-s}_{p'} (\pa \Om)$ is a dual operator of \eqref{dual}, where $\frac1p + \frac1{p'} =1$. Using \eqref{boundary integral operator}, we have
\begin{align} \label{dual2}
\notag << S^*_{2\al} \psi, \phi>> & = < \psi, S_{2\al}\phi >
  = \int_{\pa \Om} \psi(P) S_{2\al} \phi(P) dP\\
  &= \int_{\pa \Om}  \phi(P) S_{2\al} \psi(P)  dP
 =<\phi, S_{2\al} \psi>,
\end{align}
where
$<\cdot, \cdot>$ is the duality paring between
$H^{s + 2\al -1}_p (\pa \Om) $ and $H^{-s -2\al +1}_{p'}(\pa \Om), \,\, \frac1p + \frac1{p'} =1$,  and  $<<\cdot, \cdot>>$ is the duality paring between
$H^{s }_p (\pa \Om) $ and $H^{-s }_{p'}(\pa \Om), \,\, \frac1p + \frac1{p'} =1$.
Since $C^2(\pa \Om)$ is dense subset of $ H^s_p (\pa \Om)$, \eqref{dual2} implies that $S^*_{2\al}: H^{-s - 2\al +1}_{p'}(\pa \Om) \ri   H^{-s}_{p'} (\pa \Om)$ is a same operator with  $S_{2\al}: H^{-s - 2\al +1}_{p'}(\pa \Om) \ri   H^{-s}_{p'} (\pa \Om)$ for $s < 0$.

\section{Proof of Theorem \ref{theo1}}\label{bijectivity-single}
\setcounter{equation}{0}

To prove the Theorem \ref{theo1}, we need the following Lemma.
\begin{lemm}\label{lemm0916}
Let $\frac12 < \al <1$.
Given  sufficiently small $\ep > 0$, there are bounded linear operators $T^1 :L^p(\pa \Om) \ri  H^1_p(\pa \Om)$  with  $\|T^1\|_{L^p(\pa \Om) \ri H^1_p(\pa \Om)} < c\ep$
and
$T^2:H^{-1}_p(\pa \Om) \ri  H^1_p(\pa \Om)$ such that
\begin{align}\label{lemma3}
S_{2\al} S_{3 -2\al} = S_2 + T^1 + T^2.
\end{align}
\end{lemm}
\begin{rem}\label{rem0821}
\begin{itemize}

\item[(1)]
Since $S_{2}: L^p(\pa \Om) \ri H^1_p(\pa \Om)$ is bijective, for sufficiently small $\ep>0$, $S_{2}+ T^1: L^p(\pa \Om) \ri H^1_p(\pa \Om)$ is
bijective.
\item[(2)]
Since $S_2$, $S_{2\al} S_{3 -2\al}$  and $T^2$ are bounded operators from $ H^{-1}_p(\pa \Om)$ to $   L^p(\pa \Om)$, $T^1: H^{-1}_p(\pa \Om) \ri L^p(\pa \Om)$ is also bounded operator. Then, by the complex interpolation property \eqref{0310}, we get that for $-1 < s <0 $,
\begin{align}\label{1015}
\| T^1\|_{H^{s}_p(\pa \Om) \ri H^{1+s}_p(\pa \Om) } \leq c \ep^{1 +s}.
\end{align}
\item[(3)]
By the argument of  (1), (2) and by Proposition \ref{props-2}, we get $S_2 + T^1: H^s_p (\pa \Om) \ri H^{1 +s}_p(\pa \Om), \,\, -1 < s < 0$ is bijective.

\end{itemize}
\end{rem}
{\bf Proof of Lemma \ref{lemm0916}.}
Let $ 0 < 9\ep < r_0$, where $r_0>0$ is defined in a section  \ref{inpaomega}.
Let $P_1, P_2, \cdots, P_m \in \pa \Om$ such that
$|P_i -P_j| >  \ep$ and $\pa \Om \subset \cup_{1 \leq i \leq m} B(P_i, \frac32\ep)$.
Let $\{\eta_i\}_{1 \leq i\leq m}$ be a partition of unity of $\{ B(P_i, 2 \ep)\}_{1 \leq i \leq m}$ such that $supp \, \eta_i \subset B(P_i, 2 \ep), \,\, \eta_i \equiv 1 $ in $B(P_i, \frac32\ep)$. Let $\ka_i = \sum_{|P_i -P_j| \leq 5\ep} \eta_j$ and $\la_i = \sum_{|P_i -P_j| \leq 7\ep} \eta_j $ such that $supp \, \ka_i \subset  B(P_i, 7\ep)$ and $supp \, \la_i \subset  B(P_i, 9\ep)$.  Then,  for $\phi \in L^p(\pa \Om)$, we have
\begin{align*}
S_{2\al} S_{3 -2\al} \phi 
 & =  \sum_{i}     \eta_i S_{2\al} \ka_i S_{3 -2\al}  \phi
   + \sum_{i}    \eta_i S_{2\al}  ( 1-\ka_i )S_{3 -2\al}  \phi\\
  & : = I_1\phi + I_2\phi.
  \end{align*}
Note that for $P \in \De(P_i, 2\ep)$,  $ \ka_i (\cdot) \Ga_{2\al}(P - \cdot)  $ has no singularity   in $\pa \Om \setminus \De (P_i,3\ep) $ and so  $ I_2 \phi$ is a smooth function.

For $I_1 \phi$,   we have
\begin{align*}
I_1 \phi (P) & = \sum_{i}     \eta_i S_{2\al} \ka_i S_{3 -2\al} \la_i  \phi (P)
   + \sum_{i}    \eta_i S_{2\al}   \ka_i S_{3 -2\al} (1 -\la_i) \phi (P)\\
    &=    \sum_i   \eta_i (P) \int_{\pa \Om } \Ga_{2\al}(P - Z)  \ka_i(Z)    \int_{\De(P_i,9\ep) }\la_i(Q)\Ga_{3 -2\al}(Z-Q) \phi(Q)   dQ dZ\\
    & \qquad +     \sum_i    \eta_i(P) \int_{\pa \Om} \Ga_{2\al}(P - Z)  \ka_i(Z)   \int_{  \pa \Om \setminus \De (P_i,9\ep)   }(1-\la_i(Q)) \Ga_{3 -2\al}(Z-Q) \phi(Q)   dQ dZ\\
    &: =  \sum_i  I^i_{11} \phi(P) +  \sum_i I^i_{12} \phi(P)\\
    &: =  I_{11} \phi(P) + I_{12} \phi(P).
\end{align*}
For fixed $Z \in \De (P_i,7\ep)$, $ \Ga_{3 -2\al}(Z-\cdot)   $ has no singularity in $\pa \Om \setminus \De (P_i,9\ep)$  and so  $ \ka_i(Z) \int_{\pa \Om \setminus \De (P_i,9\ep)} (1- \la_i(Q) ) \Ga_{3 -2\al}(Z-Q) \phi(Q)   dQ$ is a smooth function in $\pa \Om$. Hence, by  Proposition \ref{props-2}, $  I_{12} \phi$ is a  smooth function.

Similarly, we decompose $S_2 \phi$ as
\begin{align*}
S_2 \phi(P) = J_{11}\phi(P) + J_{12} \phi(P) + J_2\phi(P),
\end{align*}
where $  J_{12} \phi, \,\,J_2\phi$ are smooth functions and
\begin{align*}
 J_{11} \phi(P) =  \sum_i    J^i_{11} \phi(P)
\end{align*}
with $J^i_{11} \phi(P) : =  \eta_i(P)    \int_{   \De (P_i,9\ep)   } \Ga_2(P-Q) \la_i(Q) \phi(Q)   dQ  $.

For $I_{11} \phi, \,\, J_{11} \phi$, we fix $i$. Using the translation and  rotation, we assume that $P_i =0$  and there is $\Psi :  B'(0, 9\ep) \ri{\mathbb R}$  with
\begin{align}\label{Phi}
|\Psi_i(x') | < c|x'|^2< c\ep^2, \,\, |\na \Psi_i(x')| < c|x'|< c\ep, \quad \mbox{for } \,\, x' \in B'(0,9\ep)
\end{align}
such that for $Q \in  \De(P_i,9\ep): = \De^i_{9\ep}$, $Q $ is represented by $Q = (y', \Psi(y'))$ for some $y' \in B'(0, 9\ep) : = B'(9\ep)$.
Let $P = (x', \Psi(x')), \,\, x' \in B'(0, 2\ep)$.
Then,  we have
\begin{align*}
& I^i_{11} \phi(P) 
   = \eta_i(P) \int_{\De^i_{9\ep}}  \la_i(Q) \phi(Q) \int_{\pa \Om} \ka_i(Z) \Ga_{2\al}(P - Z) \Ga_{3 -2\al}(Q -Z) dZ dQ\\
& =  \eta_i(P) \int_{B'(9\ep)} \la_i(y', \Psi(y'))  \phi(y', \Psi(y') )  \sqrt{1 + |\na \Psi(y')|^2}   \times \\
&   \qquad     \int_{B'(7\ep)}  \ka_i(z', \Psi(z'))\Ga_{2\al}(x' -z', \Psi(x')- \Psi(z')) \Ga_{3 -2\al}(y'- z', \Psi(y')-\Psi(z')) \sqrt{1 + |\na \Psi(z')|^2}    dz'   dy'\\
\end{align*}
and
\begin{align*}
J^i_{11} \phi(P) : =  \eta_i(P)   \int_{B'(9\ep)}   \Ga_2(x'-y', \Psi(x')-\Psi(y'))  \la_i(y', \Psi(y'))  \phi(y', \Psi(y') )   \sqrt{1 + |\na \Psi(y')|^2}  dy'.
\end{align*}
Let
\begin{align*}
I^i_{111}\phi(P):& = \eta_i(P)  \int_{B'(9\ep)}  \la_i(y', 0)   \phi(y', \Psi(y'))\int_{B'(7\ep)} \ka_i(z',0)
 \Ga_{2\al}(x' -z', 0) \Ga_{3 -2\al}(y'- z', 0)   dz'   dy',\\
J^i_{111} \phi (P):& = \eta_i(P)  \int_{B'(9\ep)} \la_i(y', 0)
 \Ga_{2}(x' -y', 0)  \phi(y', \Phi(y'))  dy'.
 \end{align*}
Note that by \eqref{Phi}, we have
\begin{align*}
\| I^i_{11} - I^i_{111}\|_{L^p (\De^i_{9\ep}) \ri H^1_p(\De^i_{2\ep})}, \qquad  \| J^i_{11} - J^i_{111}\|_{L^p (\De^i_{9\ep}) \ri H^1_p(\De^i_{2\ep})} \leq c \ep.
\end{align*}
(See \cite{CMM}).

It is well-known
\begin{align*}
 \int_{\Rn} \Ga_{2\al} (x'-z', 0) \Ga_{3 -2\al} (y' -z', 0) dz' =  \Ga_2(x' - y',0).
 \end{align*}
(See Section 5.1 in \cite{St}.)
 Hence, we have
 \begin{align*}
  \Ga_2(x' - y',0)& = \int_{\Rn} \ka_i(z',0)
 \Ga_{2\al}(x' -z', 0) \Ga_{3 -2\al}(y'- z', 0)   dz'\\
 & \qquad  +  \int_{\Rn}   (1 - \ka_i(z',0) )
 \Ga_{2\al}(x' -z', 0) \Ga_{3 -2\al}(y'- z', 0)   dz'\\
 & = \int_{\Rn} \ka_i(z',0)
 \Ga_{2\al}(x' -z', 0) \Ga_{3 -2\al}(y'- z', 0)   dz' +  k_i(x', y').
 \end{align*}
 Hence, we have that
 \begin{align*}
&  I^i_{111}\phi(P) - J^i(111) \phi(P)  =    \eta_i(P) \int_{B'(9\ep)} \la_i(y', 0)  \phi(y', \Psi(y') )     k _i(x', y')  dy'
 \end{align*}
is a smooth function. Let $T^1 =  \sum_i(I^i_{11} - I^i_{111}) + \sum_i(J^i_{11} - J^i_{111}) $ and $T_2 = I_2 + J_2 + I_{12} + J_{12} + \sum_i( I^i_{111}  - J^i(111) ) $.
Then,  $ S_{2\al} S_{3 -2\al} = S_2 + T^1 + T^2$ such that $T^2$ is a smooth function and
\begin{align*}
 \| T^1\phi \|_{ H^1_p(\pa \Om) } & \leq c \sum_i \big(\| (I^i_{11} -I^i_{111}) \phi\|_{H^1_p (\pa \Om)} +   \|(J^i_{11} -J^i_{111}) \phi\|_{H^1_p (\pa \Om)} \big)\\
 & \leq c \ep \sum_i \| \phi\|_{L^p(\De^i_{9\ep})} \leq c \ep \| \phi\|_{L^p(\pa \Om)}.
\end{align*}
Hence, we complete the proof of Lemma \ref{lemm0916}. $\Box$

{\bf Proof   Theorem \ref{theo1}.}

{\Large (1).  In the case of $p \geq  p_0: = \frac{2(n-1)}{n-2 + 2\al}$ ($p_0 < 2$).}
To show the injectivity, suppose that $S_{2\al} \phi =0$ for $\phi \in L^p (\pa \Om)  $. By the H$\ddot{o}$lder inequality  and Sobolev imbedding,
  we have $L^p (\pa \Om) \subset  L^{p_0} (\pa \Om) \subset H^{-\al +\frac12}_2(\pa \Om)$. Hence, by  the bijectivity of   $S_{2\al}:  H^{-\al +\frac12}_2(\pa \Om) \ri  H^{\al -\frac12}_2(\pa \Om)$ (see \cite{C}), we have $\phi =0$. Hence, $S_{2\al}:  L^p(\pa \Om) \ri  H^{2\al -1}_p(\pa \Om)$ is bijective for $ p \geq p_0 $.

To show that $S_{2\al}:  L^p(\pa \Om) \ri  H^{2\al -1}_p(\pa \Om)$ is surjective,  let $f \in   H^{2\al -1 }_p (\pa \Om)$.
By the Sobolev imbedding and the H$\ddot{o}$lder inequality, we have $H_p^{2\al -1 } (\pa \Om) \subset H_{p_0}^{2\al -1 } (\pa \Om) \subset H_2^{\al - \frac12}(\pa \Om)$. By the bijectivity of  $S_{2\al} : H^{-\al +\frac12}_2(\pa \Om) \ri H^{\al -\frac12 }_2(\pa \Om)$,
there is $\phi \in H^{-\al +\frac12}_2(\pa \Om)$ such that $S_{2\al} \phi = f$.
Note that by the Lemma \ref{lemm0916}, we get that $S_{3 -2\al} S_{2\al} = S_2 + T^1 + T^2$, where $\| T^1\|_{L^p(\pa \Om) \ri H^1_p(\pa \Om)} \leq \ep $ and $T^2:H^{-1}_2(\pa \Om) \ri H^1_2(\pa \Om)$ is bounded. Then, by   the Lemma \ref{lemm0916}, we obtain that
$(S_2 + T^1)\phi   = S_{3-2\al} S_{2\al} \phi  - T^2\phi \in H^1_p(\pa \Om)$.  Taking $\ep>0$ sufficiently small such that
  $S_2 + T^1: L^p(\pa \Om) \ri H^1_p(\pa \Om)$ is bijective (see (1) of the Remark \ref{rem0821}), we obtain that $\phi \in L^p(\pa \Om)$. This implies that $S_{2\al} : L^p(\pa \Om) \ri H^{2\al -1 }_p (\pa \Om)$ is surjective. Hence, we complete the proof of  the bijectivity of $S_{2\al}: L^p(\pa \Om) \ri H^{2\al -1}_p(\pa \Om)$ for  $p \geq \frac{2(n-1)}{n-2 + 2\al} $.
$\Box$
\begin{rem}\label{rem1021}
\begin{itemize}
\item[(1)]
 Note that the dual operator $S_{2\al}^* : H^{-2\al +1}_{p'} (\pa \Om) \ri L^{p'}(\pa \Om)$ of $ S_{2\al} : L^p(\pa \Om) \ri H^{2\al -1 }_p (\pa \Om)$ are same with the operator $S_{2\al} : H^{-2\al +1}_{p'} (\pa \Om) \ri L^{p'}(\pa \Om)$, where $\frac1p + \frac1{p'} =1$ by the section \ref{operator}.  Hence, by the property of the  dual operator, we get that $S_{2\al}: H^{-2\al +1}_p (\pa \Om) \ri L^p(\pa \Om)$ is bijective for $ 1 <p \leq \frac{2(n-1)}{n-2\al}= p_0^{'}$.
\item[(2)]
In the Lemma \ref{lemm0916}, $ S_{3 -2\al}S_{2\al}$ is sum of a bijective operator $S_2 + T^1$ and  a compact operator $T^2$ and so $S_{3-2\al} S_{2\al}$ is the Fredholm operator with index zero. Since  $S_{2\al}: L^p(\pa \Om) \ri H^{2\al -1}_p(\pa \Om)$, $S_{3 -2\al}: H^{2\al -1}_p(\pa \Om) \ri H^1_p(\pa \Om)$ are injective,  $ S_{3 -2\al}S_{2\al}$ is injective and so by the Fredholm operator theorem, $ S_{3 -2\al}S_{2\al}  : L^p(\pa \Om) \ri H^1_p(\pa \Om)$ is bijective. This implies that $S_{3 - 2\al}: H^{2\al -1}_p (\pa \Om) \ri H^1_p (\pa \Om)$ is bijective for $p \geq \frac{2(n-1)}{(n -2 + 2\al)}$.

\end{itemize}
\end{rem}


%
%
%
%

{\Large (2).  In the case of $1< p< 2$.}
Now, we will show that $S_{2\al}:  H^{-2\al +1}_q(\pa \Om) \ri  L^q(\pa \Om)$ is surjective for $q = \frac{p}{p-1} > 2$. Let $f \in L^q(\pa \Om)$.
By the H$\ddot{o}$lder inequality, $L^q (\pa \Om) \subset L^2 (\pa \Om)$ and by the bijectivity of  $S_{2\al} : H^{-2\al +1 }_2(\pa \Om) \ri L^2(\pa \Om)$ (see (1) of Remark \ref{rem1021}),
there is $\phi \in H^{-2\al +1}_2(\pa \Om)$ such that $S_{2\al} \phi = f$.  Then,
\begin{align*}
S_{3-2\al} S_{2\al} \phi = S_{3 -2\al} f \in H_q^{2-2\al }(\pa \Om).
\end{align*}
By  (3) of  Remark \ref{rem0821}, we obtain that
$(S_2 + T^1 )\phi \in H^{2-2\al}_q(\pa \Om)$.
Since $S_2 + T^1: H^{-2\al +1}_q(\pa \Om) \ri H^{2-2\al}_q(\pa \Om)$ is bijective (see (3) of  Remark \ref{rem0821}), we obtain that $\phi \in H^{-2\al +1}_q(\pa \Om)$. This implies that $S_{2\al} : H^{-2\al +1}_q(\pa \Om) \ri L^q (\pa \Om)$ is surjective.

By the property of the dual operator, we have that $S^*_{2\al}: L^p (\pa \Om) \ri H^{2\al -1}_p(\pa \Om), \, 1 < p < 2$ are injective.
Since $S_{2\al}^* = S_{2\al}$, $ S_{2\al}: L^p (\pa \Om) \ri H^{2\al -1}_p(\pa \Om) $ is injective and so  $S_{2\al}:H^{2\al -1}_p (\pa \Om) \ri H^{1}_p(\pa \Om)$ is injective.
Hence, we get $S_{2\al} S_{3 -2\al}: L^p(\pa \Om) \ri H^1_p(\pa \Om), \, 1 < p<2$ are injective.

Note that in (3) of   Remark \ref{rem0821},  $S_{2\al} S_{3-2\al}$ is the sum of a bijective operator and a compact operator. Hence, by the Fredholm theorem, $S_{2\al} S_{3 -2\al}: L^p(\pa \Om) \ri H^1_p(\pa \Om)$ is bijective.

To show that $S_{2\al}: L^p(\pa \Om) \ri H^{2\al -1}_p(\pa \Om)$ is surjective for $1 < p < 2$, let us $f \in H^{2\al -1}_p(\pa \Om)$. Then, we have
$S_{3- 2\al }f \in H^1_p(\pa \Om)$. Since, $S_{3 -2\al}S_{2\al}: L^p(\pa \Om) \ri H^1_p(\pa \Om)$ is bijective, there is $\phi \in L^p(\pa \Om)$ such that
$S_{3 -2\al}S_{2\al} \phi = S_{3-2\al} f$. Since $S_{3-2\al}$ is injective, we get $S_{2\al} \phi =f$ and so $S_{2\al} : L^p(\pa \Om) \ri H^{2\al -1}_p(\pa \Om)$ is bijective.
$\Box$

\begin{coro}\label{coro1016}
Let $ \frac12 < \al < 1$ and $1 < p < \infty$.
For $-2\al +1 \leq s \leq 2 -2\al$,
\begin{align*}
S_{2\al} : H^s_p (\pa \Om )\ri H^{s + 2\al -1}_p (\pa \Om),\\
S_{2\al} : B^s_p (\pa \Om )\ri B^{s + 2\al -1}_p (\pa \Om)
\end{align*}
are bijective.

\end{coro}
\begin{proof}
In the proof of the Theorem \ref{theo1}, we have $S_{3 -2\al} : L^p(\pa \Om) \ri H_p^{2 -2\al}(\pa \Om), \,\, S_{2\al}: H^{2 -2\al}_p (\pa \Om) \ri H_p^1(\pa \Om)$ are injective and so $S_{2\al} S_{3 -2\al}  : L^p(\pa \Om) \ri H_p^1(\pa \Om)$ is injective. Since $S_{2\al} S_{3 -2\al}  : L^p(\pa \Om) \ri H_p^1(\pa \Om)$ is Fredholm operator with index zero, we get $S_{2\al} S_{3 -2\al}  : L^p(\pa \Om) \ri H_p^1(\pa \Om)$ is bijective. This implies
 \begin{align}\label{1016-1}
 S_{2\al}: H^{2 -2\al}_p (\pa \Om) \ri H_p^1(\pa \Om) \quad \mbox{     is bijective.}
 \end{align}
 By the dual operator property and the fact of $S^*_{2\al} = S_{2\al}$, we have
 \begin{align}\label{1016-2}
 S_{2\al}:  H^{-1}_p (\pa \Om) \ri H^{-2 + 2\al}_p(\pa \Om) \quad \mbox{ is bijective.}
  \end{align}
  Using \eqref{1016-1}, \eqref{1016-2} and  the properties of the real interpolation and complex interpolation, we obtain  the corollary.

\end{proof}

\section{Proof of Theorem \ref{theo0905} }\label{proof-maintheo}
\setcounter{equation}{0}

We introduce a Riesz potential $I_{2\al}$, $0  < 2\al < n$, by
$$
I_{2\al} f(x) = c(n,\al) \int_{\R} \frac{1}{|x -y|^{n-2\al}}
f(y) dy\mbox{ for }\psi
\in C^\infty_c (\R),
$$
where $c(n,\al) : = \frac{ (2\pi)^{2\al} \Ga(\frac{n}2 -\al) }{ \pi^{\frac{n}2}  2^{2\al}\Ga(\al ) }$.

The following proposition is well known fact and will be useful in the subsequent estimates (see chapter 5 of \cite{St}).
\begin{prop}\label{prop2}
\begin{itemize}

\item[1).] Let $1 < p < q < \infty, \,\, \frac1q = \frac1p -
\frac{2\al}{n}$. Then
$$
I_{2\al}:L^p (\R) \ri L^q(\R)
$$
is bounded.

\item[2).]  Let $1 < p < \infty $ and $s \in {\mathbb R}  $. Then, the following operators are bounded.
\begin{align*}
I_{2\al} : H^s_p (\R) \ri H^{s +2\al}_p (\R), \qquad I_{2\al} : B^s_p (\R) \ri B^{s +2\al}_p (\R).
\end{align*}

\end{itemize}
\end{prop}

\begin{rem}\label{rem1027}
Let $B_R$ be the open ball in $\R$ centered at the  origin with radius $R$.
Then, by proposition \ref{prop2}, the following operator is bounded.
\begin{align*}
\tilde I_{2\al}: B^s_{p0}(B_R) \ri B^{s + 2\al}_p (B_R), \quad \tilde I_{2\al}
f(x)
=  \int_{\R} \Ga_{2\al}(x-y)  f(y) dy, \quad f \in B^s_{p0}(B_R),\qquad s \in {\mathbb R}.
\end{align*}
\end{rem}

{\bf Proof of \eqref{0904}}
Let $ -2\al +1 -\frac1p < s <0$.
Let $\phi \in C^1(\pa \Om) $ and $f \in C^\infty_c (B_{R})$. Then, we have
\begin{align*}
\int_{\R}  f(x) {\mathcal S}_{2\al} \phi (x) dx  = \int_{\pa \Om}  \phi(P) \tilde I_{2\al} f (P) dP.
\end{align*}
Since $C^1(\pa \Om) $ is dense subspace of $ B^{s}_p (\pa \Om)$ and $C^\infty_c (B_{R})$ is dense subspace of $B^{-s -2\al+\frac1q}_{q0} (B_{R})$, we have
\begin{align*}
< f, {\mathcal S}_{2\al} \phi>_{(  B^{-s -2\al+\frac1q}_{q0} (B_{R}), B^{s + 2\al -1 + \frac1p}_p(B_{R}))} = < \phi, \tilde I_{2\al} f>_{(  B^{ s}_p (\pa \Om), B^{-s }_q(\pa \Om))},
\end{align*}
where  $\frac1p + \frac1q =1$.  Then, we have
\begin{align*}
< f, {\mathcal S}_{2\al} \phi>_{(  B^{-s -2\al+\frac1q}_{q0} (B_{R}), B^{s + 2\al -1 + \frac1p}_p(B_{R}))}
& \leq \| \phi\|_{B^{s}_p (\pa \Om)} \| \tilde I_{2\al} f\|_{B^{-s}_q (\pa \Om)}\\
& \leq c\| \phi\|_{B^{s}_p (\pa \Om)} \| \tilde I_{2\al} f\|_{B^{-s+\frac1q}_q (B_R)}\\
& \leq c\| \phi\|_{B^{s}_p (\pa \Om)} \|   f\|_{B^{-s -2\al+\frac1q}_{q0} (B_R)}.
\end{align*}
Hence, we have
\begin{align*}
\| {\mathcal S}_{2\al} \phi\|_{ B^{s + 2\al -1 + \frac1p}_p(B_{R})}
& =\sup_{ \|   f\|_{B^{-s -2\al+ 1 -\frac1p}_{q0} (B_{R})} =1} |< f, {\mathcal S}_{2\al} \phi>_{(  B^{-s -2\al+\frac1q}_q (B_R), B^{s + 2\al -1 + \frac1p}_p(B_R))}|\\
& \leq c\| \phi\|_{B^{s}_p (\pa \Om)}.
\end{align*}
We complete the proof of  \eqref{0904}. $\Box$

{\bf Proof of   \eqref{0904-2}}.
For $\phi \in B^s_p(\pa \Om), \, -2\al +1 -\frac1p < s < 0$, let  us $u$ be a the layer potential
 of $\phi$  defined by \eqref{def-single-dual}.
Note  that $u$ is in $C^\infty(\R \setminus \pa
\Om)$   and for  large $|x|$, we have
\begin{align}\label{behavior2}
 |D^{\be}u(x)|
 \leq \|\phi\|_{B^s_p (\pa \Om)} \|D^\be \Ga_{2\al} (x -
\cdot)\|_{B^{-s}_{p'}(\pa \Om)} \leq c \|\phi\|_{B^s_p (\pa \Om)}
\frac{1}{|x|^{n-2\al+ |\be|}},
\end{align}
where $\frac1p +\frac{1}{p'} =1$.

Let $B_R$ be an open ball whose center is origin and radius is $R \geq 2$ such that $\Om \subset B_{\frac13 R}$.
We divide the left-hand side of \eqref{0904-2} into three parts
\begin{align}\label{homo-divide}
\begin{array}{ll}\vspace{2mm}
 A_1 & = \int_{|x| \leq R} \int_{|y| \leq R}  \frac{|D_x^k u(x ) - D_y^k u( y)|^p}{| x- y|^{n+p(s + 2\al -k-1 +\frac1p)}} dydx,\\ \vspace{2mm}
 A_2  &=  2\int_{|x| \leq R} \int_{|y| \geq  R} \frac{|D_x^k u(x ) -  D_y^k u( y)|^p}{| x- y|^{n+p(s + 2\al -k-1  +\frac1p)}} dydx, \\
  A_3 &= \int_{|x| \geq R} \int_{|y| \geq R} \frac{|D_x^k u(x ) - D_y^k u( y)|^p}{| x- y|^{n+p(s + 2\al -k-1 +\frac1p)}} dydx.
\end{array}
\end{align}
By \eqref{0904}, $A_1$  is dominated by $ \| \phi \|^p_{B_p^s (\pa \Om)}$.  For $|x| \leq R$ and $|y| \geq 2R$, we get that $|x-y| \geq |y| -|x| \geq |y| -R\geq \frac12 |y|  $.
Note that  by \eqref{behavior2}, for $|y| \geq 2R$,  we have  that $|D_y^k u(y)| \leq c  |y|^{-n +2\al -k}\| \phi \|^2_{ B_p^s(\pa \Om)}$.  Hence, by  \eqref{0904}, we have
\begin{align*}
A_2 & \leq   2\int_{|x| \leq R} \int_{ R \leq |y| \leq  2R} \frac{|D_x^k u(x ) - D_y^k u( y)|^p}{| x- y|^{n+p(s + 2\al -k -1 +\frac1p)}} dydx\\
& \qquad + 2^{n+2s +2}\int_{|x| \leq R} \int_{|y| \geq 2R} \frac{|D_x^k u(x ) |^p +  |D_y^k u( y)|^p}{| y|^{n+p(s + 2\al -k -1 +\frac1p)}} dydx \\
& \leq c_R \| u\|^p_{B^{s + 2\al -1 +\frac1p}_p(B(2R)) }   + c \| \phi \|^p_{B_p^s(\pa \Om)}  \int_{|x| \leq R}
\int_{|y| \geq 2R} \frac{  1}{|  y|^{n+p(s -1 +\frac1p +n)}} dydx\\
& \leq c_R \| \phi\|^p_{B^s_p(\pa \Om) }.
\end{align*}
We divide $A_3$ into two parts;
\begin{align}\label{homo-divide2}
\begin{array}{ll}\vspace{2mm}
&\int_{|x| \geq R} \int_{|y| \geq R, |x-y| \leq \frac12 |x|}  \frac{|D_x^k u(x ) -  D_y^k u( y)|^p}{| x- y|^{n+p(s + 2\al -k -1 +\frac1p)}} dydx\\
&\qquad + \int_{|x| \geq R} \int_{|y| \geq R, |x-y| \geq \frac12 |x|}  \frac{|D_x^k u(x ) -  D_y^k u( y)|^p}{| x- y|^{n+p(s + 2\al -k -1 +\frac1p)}} dydx.
\end{array}
\end{align}
Applying the mean-value theorem, for $|x| \geq R, \,\, |x-y| \leq \frac12 |x|$,  there is a $\xi $ between $x$ and $y$
such that $D_x^k u(x) - D_y^k u(y) = D^{k+1} u(\xi) \cdot (x-y)$. Note that $|x-\xi| \leq \frac12 |x|$  and hence $|\xi| \geq \frac12 |x| \geq \frac12 R$.
Since $s +2\al -k-2 +\frac1p < 0$ and $p > \frac{n-1}{n+s -1}$, by \eqref{behavior2},  the first term of \eqref{homo-divide2} is dominated by
\begin{align*}
  &\int_{|x| \geq R} \int_{|y| \geq R,  |x-y| \leq \frac12 |x|}  \frac{|D^{k+1}u(\xi )|^p }{| x- y|^{n+p(s + 2\al -k-1 +\frac1p)-p}} dydx \\
& \leq  c\|\phi \|^p_{B^s_p(\pa \Om)} \int_{|x| \geq R} \frac{1}{|x|^{pn -2p\al +(k +1)p}}
\int_{ |x-y| \leq \frac12 |x|}  \frac{1}{| x- y|^{n+p(s + 2\al -k -2 +\frac1p)}} dydx \\
& \leq c \|\phi \|^p_{B^s_p (\pa \Om)} \int_{|x| \geq R} \frac{1}{|x|^{p(n+s -1) +1}} dx\\
& = cR^{-p(n+s -1) -1 +n}  \|\phi \|^p_{ B^s_p(\pa \Om)}.
\end{align*}
Since $|x|, \,\, |y| \geq R$, by \eqref{behavior2}, the second term
of \eqref{homo-divide2} is dominated by
\begin{align}\label{homo-divide3}
\begin{array}{ll} \vspace{2mm}
&  \int_{|x| \geq R} \int_{|y| \geq R, |x-y| \geq \frac12 |x|}  \frac{|D_x^k u(x )|^p + |D_y^k u( y)|^p}{| x- y|^{n+p(s + 2\al -k -1 +\frac1p)}} dydx \\ \vspace{2mm}
& \leq  \|\phi \|^p_{B^s_p (\pa \Om)} \int_{|x| \geq R}   \frac{1}{|x|^{p(n -2\al +k)}}
\int_{|y| \geq R, |x-y| \geq \frac12 |x|} \frac{1}{|x-y|^{n+ p(s + 2\al -k -1 +\frac1p)}} dydx\\
 & \quad +  \|\phi \|^p_{B^s_p(\pa \Om)}
\int_{|x| \geq R}  \int_{|y| \geq R, |x-y| \leq \frac12 |x|}    \frac{1}{|x-y|^{n+p(s + 2\al -k -1 +\frac1p)}} \frac{1}{|y|^{p(n -2\al +k)}}dydx.
\end{array}
\end{align}
Since $p> \frac{n-1}{n+s -1}$, the second term  of right-hand side of \eqref{homo-divide3} is dominated by $R^{-p(n+s -1) -1 +n} \|\phi \|^p_{B^s_p(\pa \Om)}$.
Note that
\begin{align*}
&\int_{|x| \geq R}  \int_{ R \leq |y| \leq 2 |x|}    \frac{1}{|x|^{n+p(s + 2\al -k-1 +\frac1p)}} \frac{1}{|y|^{pn -2p\al + kp}}dydx\\
& \leq  c  \left\{ \begin{array}{ll} \vspace{2mm}
 R^{-pn + 2p\al +n} \int_{|x| \geq R}    \frac{1}{|x|^{n+p(s + 2\al -k -1 +\frac1p)}} dx, &\quad pn -2p\al + kp > n,\\ \vspace{2mm}
\int_{|x| \geq R}    \frac{\ln |x|}{|x|^{n+ p(s + 2\al -k -1 +\frac1p)}}  dx, & \quad pn -2p\al + kp =n,\\
\int_{|x| \geq R}    \frac{1}{|x|^{p(n+s -k -1) +1}} dx, & \quad pn -2p\al + kp <n
\end{array}
\right.\\
& \leq c R^{-p(n -k -1 +s) -1 +n}  \ln R.
\end{align*}
Then, since $p > \frac{n-1}{n +s-1}$, the first  term  of right-hand side of \eqref{homo-divide3} is dominated by
\begin{align*}
&   \|\phi \|^p_{B^s_p(\pa \Om)} \int_{|x| \geq R}  \int_{|y| \geq R, |x-y| \geq \frac12 |x|}    \frac{1}{|x-y|^{n+p(s + 2\al -k -1 + \frac1p) }} \frac{1}{|y|^{pn -2p\al + kp}}dydx \\
&  \leq c  \|\phi \|^p_{ B^s_p(\pa \Om)}
           \Big( \int_{|x| \geq R}  \int_{ R \leq |y| \leq  2|x|}    \frac{1}{|x|^{n+p(s + 2\al -k -1 + \frac1p)}} \frac{1}{|y|^{pn -2p\al + kp}}dydx\\
& \qquad + \int_{|x| \geq R}  \int_{|y| \geq 2 |x| }     \frac{1}{|y|^{n + p(n +s -1) +1}}dydx \Big) \\
&  \leq c R^{-p(n -1 +s) -1 +n}  \ln R  \|\phi \|^p_{B^s_p (\pa \Om)}.
\end{align*}
Therefore, we  showed  that $ A_1 + A_2 + A_3  \leq c_R  \|\phi \|^p_{B^s_p (\pa \Om)}$ and hence showed  \eqref{0904-2}.
$\Box$

\section{Proof of Theorem \ref{theo2-1} }\label{proof-maintheo-2}
\setcounter{equation}{0}

\begin{theo}\label{rem2}
Let $1-2\al  -\frac1p< s <0$.
For $\phi \in B^s_p (\pa \Om)$, let us
$u = {\mathcal S}_{2\al} \phi$ be a layer potential defined in
\eqref{def-single-dual}. The, we have
\begin{itemize}
\item[(1)]
\begin{align}\label{Fourier}
\hat{u}(\xi) = |\xi|^{-2\al} < \phi,  e^{2\pi i \xi \cdot  }>.
\end{align}
\item[(2)]
$u$ is a weak solution of
\begin{align}\label{solution}
\De^\al u =0, \quad \mbox{in} \quad \R \setminus \pa \Om.
\end{align}

\end{itemize}
\end{theo}

\begin{proof}
For the proof of   \eqref{Fourier}, let us
$\phi \in C^2(\pa \Om)$
and $\psi \in C^\infty_c (\R)$.
Then, we have
\begin{align*}
\int_{\R} u (x) \psi(x) dx  &= c(n,s)\int_{\pa \Om} \phi (Q)
\int_{\R}
\frac{1}{|x-Q|^{n-2\al}} \psi(x) dx dQ\\
 & =   \int_{\pa \Om} \phi (Q) \int_{\R} |\xi|^{-2\al} e^{2\pi i \xi \cdot Q}
 \overline{\hat{\psi}(\xi)}
 d\xi  dQ\\
 & =  \int_{\R} \overline{\hat{\psi}(\xi)} |\xi|^{-2\al} \int_{\pa \Om} \phi(Q) e^{2\pi
 i \xi \cdot Q} dQd\xi.
\end{align*}
Hence, we get
\begin{align*}
\hat{u}(\xi) = |\xi|^{-2\al} \int_{\pa \Om} \phi(Q)  e^{2\pi i
\xi \cdot Q } dQ.
\end{align*}
Since $C^2 (\pa \Om)$ is dense in $B^s_p(\pa \Om)$, we get \eqref{Fourier}
for all  $\phi \in B^s_p(\pa \Om)$.

For the proof \eqref{solution}, suppose that $\phi \in C^2 (\pa \Om)$ and
$\psi \in C^\infty(\R \setminus \pa \Om)$, then, by \eqref{Fourier},
\begin{align}\label{0929}
\notag \int_{\R}   u(x) \De^\al \psi(x) dx & = \int_{\R} |\xi|^{2\al} \hat {u}(\xi) \overline{\hat{\psi}(\xi)} d\xi \\
\notag &=  \int_{\R} \overline{\hat{\psi} (\xi)} 
     \int_{\pa \Om} e^{-2\pi i \xi \cdot Q}  \phi(Q) dQ
 d\xi\\
\notag & =   \int_{\pa \Om} \phi (Q)  \overline{\int_{\R} e^{2\pi i \xi \cdot Q }
 \hat{\psi} (\xi) d\xi} dQ\\
\notag & = \int_{\pa\Om} \phi(Q) \psi (Q) dQ\\
 & =0
\end{align}
Since $\De^t: \dot B_p^s (\R) \ri \dot B_p^{s -2t}(\R) $ is isomorphism, we have
\begin{align*}
|\int_{\R}   u(x) \De^\al \psi(x) dx|  & = |\int_{\R}\De^{\frac12(s + 2\al -1 +\frac1p)} u(x) \De^{\frac12(-s + 1 -\frac1p)} \psi (x) dx|\\
& \leq \| \De^{\frac12(s + 2\al -1 +\frac1p)} u\|_{\dot B^0_p(\R)} \| \De^{\frac12(-s + 1 -\frac1p)} \psi \|_{ \dot  B^0_{p'}(\R)}\\
& \leq \| u\|_{\dot B^{s + 2 \al -1 +\frac1p}_p(\R)} \| \psi \|_{ \dot B^{-s + 1 -\frac1p}_{p'}(\R)}\\
& \leq c\| \phi\|_{B^s_p(\pa \Om)} \| \psi \|_{ B^{-s + 1 -\frac1p}_{p'}(\R)}.
\end{align*}
Let $\phi_k \in C^2 (\pa \Om)$ such that $\phi_k \ri \phi$ in $B_p^s(\pa \Om)$ and $u_k = {\mathcal S} \phi_k$. Then, we get
\begin{align*}
|\int_{\R}  (  u_k(x) - u(x)  ) \De^\al \psi(x) dx | \leq c  \| \phi_k -\phi\|_{B^s_p(\pa \Om)} \| \psi \|_{ B^{-s + 1 -\frac1p}_{p'}(\R)} \ri 0, \quad k \ri \infty.
\end{align*}
Hence, since  $C^2 (\pa \Om)$ is the dense subspace of $B^s_p(\pa \Om)$, \eqref{0929} holds for $\phi \in  B^s_p(\pa \Om)$ and so
we get  \eqref{solution} for all  $\phi \in B^s_p(\pa \Om)$.
\end{proof}

{\bf Proof of Theorem \ref{theo2-1}.}
 By the  Corollary \ref{coro1016}, 
we get that $S_{2\al}: B^{t -2\al +1}_p(\pa \Om)\ri B^t_p(\pa \Om), \,  0 < t <1, \, 1 < p < \infty$ are bijective.

To show the existence of solution, let   $g \in B^t_p(\pa \Om)$.
By the bijectivity of $S_{2\al}: B_p^{t -2\al + 1}(\pa \Om) \ri B^t_p(\pa \Om)$, there is a
$\phi \in B_p^{t -2\al +  1}(\pa \Om)$ such that $S_{2\al} \phi = g$. Let  $u = {\mathcal S}_{2\al} \phi$ defined by \eqref{def-single-dual}.
Then, by (2) of the Theorem \ref{rem2}, $u$ is a weak solution of \eqref{solution} and  by  the Theorem \ref{theo0905},  $u$ satisfies the equation \eqref{main result}. Hence, we complete the proof of the Theorem \ref{theo2-1}.
$\Box$

\end{document}